\documentclass[12pt]{article}

\usepackage{amsmath}\usepackage{amssymb}\usepackage{amsfonts}\usepackage{amsthm}
\usepackage{color}
\usepackage{setspace}
\usepackage{graphicx}
\usepackage[english]{babel}\usepackage[latin1]{inputenc}\usepackage{times}\usepackage[T1]{fontenc}
\usepackage{epstopdf}
\usepackage{epsfig}
\usepackage{float}

\newtheorem{thm}{\bf{Theorem}}[section]
\newtheorem{lem}[thm]{\bf{Lemma}}

\newtheorem{df}[thm]{\bf{Definition}}
\newtheorem{cor}[thm]{\bf{Corollary}}

\newtheorem{prop}[thm]{\bf{Proposition}}
\newtheorem{fact}[thm]{\bf{Fact}}
\newtheorem{ex}[thm]{\bf{Example}}

\newcommand{\To}{\ensuremath{\rightrightarrows}}
\newcommand{\menge}[2]{\big\{{#1} \mid {#2}\big\}}

\numberwithin{equation}{section}

\newcommand{\dom}{\operatorname{dom}}
\newcommand{\intt}{\operatorname{int}}

\newcommand{\ri}{\operatorname{ri}}

\newcommand{\epi}{\operatorname{epi}}

\newcommand{\ran}{\operatorname{ran}}
\newcommand{\Id}{\operatorname{Id}}

\newcommand{\elimsup}{\operatornamewithlimits{elimsup}}
\newcommand{\eliminf}{\operatornamewithlimits{eliminf}}

\newcommand{\argmin}{\operatornamewithlimits{argmin}}

\newcommand{\Prox}{\operatorname{Prox}}

\newcommand{\bdry}{\operatorname{bdry}}

\newcommand{\R}{\operatorname{\mathbb{R}}}
\newcommand{\RX}{\operatorname{\overline{\R}}}
\newcommand{\N}{\operatorname{\mathbb{N}}}
\newcommand{\B}{\operatorname{\mathbb{B}}}

\title{Strongly convex functions, Moreau envelopes and the generic nature of convex functions with strong minimizers\\}
\author{C. Planiden\thanks{Mathematics, University of British Columbia Okanagan, Kelowna, B.C. V1V 1V7, Canada. Research by this author was supported by UBC UGF and by NSERC of Canada. chayne.planiden@alumni.ubc.ca.}\and X. Wang\thanks{Mathematics, University of British Columbia Okanagan, Kelowna, B.C. V1V 1V7, Canada. Research by this author was partially supported by an NSERC Discovery Grant. shawn.wang@ubc.ca.}}
\date{\today}
\begin{document}

\maketitle\author
\setcounter{page}{1}\pagenumbering{arabic}

\begin{abstract}
In this work, using Moreau envelopes, we define a complete metric
for the set of proper lower semicontinuous convex functions. Under this metric, the convergence of each sequence of
convex functions is epi-convergence. We show that the set of strongly convex functions is dense but it is only of the first category. On the other hand, it is shown that the set of
convex functions with strong minima is of the second category.
\end{abstract}



%%%%%
\section{Introduction}\label{sec:intro}
%%%%%
Minimizing convex functions is fundamental in optimization, both in theory and in the algorithm design.
For most applications, the assertions that can be made about a class of convex functions are of greater value than those concerning a particular problem. This theoretical analysis is valuable for the insights. Our main result in this paper states
that the set of all proper lower semicontinuous (lsc) convex functions which have strong minimizers is of second category.
Studying strong minima is important, because
numerical methods usually produce asymptotically minimizing sequences, we can
assert convergence of asymptotically
minimizing sequences when the function has a strong minimizer.
The strongly convex function is also of great use in optimization problems, as it can significantly
increase the rate of convergence of first-order methods such as projected subgradient descent \cite{simpler}, or
more generally the forward-backward algorithm \cite[Example 27.12]{convmono}.
Although every strongly convex function has a strong minimizer, we show that
the set of strongly convex functions is only of the first category.

As a proper lsc convex function allows infinity values, we propose to relate the
function to its Moreau envelope.
The importance of the Moreau envelope in optimization is clear; it is a regularizing (smoothing) function \cite{moreau1963,proximite}, and in the convex setting it has the same local minima and minimizers as its objective function \cite{funcanal,rockwets}.

The key tool we use is Baire category. A property is said to be generic if it holds for a second category set.
We will work in a metric space defined by Moreau envelopes. In this setting,
there are many nice properties of the set of Moreau envelopes of proper, lsc, convex functions. This set is proved to be closed and convex. Moreover, as a mapping from the set of proper lsc convex functions to the set of
Moreau envelopes of convex functions, the Moreau envelope mapping is bijective. We provide a detailed analysis of functions with strong minima, strongly convex functions, and their Moreau envelopes.

\par The organization of the present work is the following. Section \ref{sec:prelim} contains notation and definitions, as well as some preliminary facts and lemmas about Baire category, epi-convergence of convex functions, strongly convex functions and strong minimizers
that we need to prove the main results. We show that the Moreau envelope of a convex function inherits many nice properties of
 the convex function, such as coercivity and strong convexity.
In Section \ref{sec:metric}, using Moreau envelopes of convex functions, we propose to use Attouch-Wets' metric
on the set of proper lsc convex functions. It turns out that
this metric space is complete, and it is isometric to the metric space of Moreau envelopes endowed with uniform convergence
on bounded sets. The main results of this paper are presented in Section \ref{sec:main}. We give some characterizations of
strong minimizers of convex functions, that are essential for our Baire category approach.  We establish
 Baire category classification of the sets of strongly convex functions, convex functions with strong minima, and convex coercive functions.  Our main result says that most convex functions have strong
minima, which in turn implies that the set of convex functions not having strong minimizers is small.
Surprisingly, the set of strongly convex functions is only of the first category.
In addition, we show that a convex function is strongly convex if and only if its
 proximal mapping is a down-scaled proximal mapping.
 Concluding remarks and areas of future research are mentioned in Section \ref{sec:conc}.

A comparison to literature is in order. In \cite{mostmax}, Baire category theory was used to show that most (i.e. a generic set) maximally monotone operators have a unique zero. In \cite{pwang2016}, a similar track was taken, but it uses the perspective of proximal mappings in particular, ultimately proving that most classes of convex functions have a unique minimizer. The technique of
this paper differs in that it is based on functions. We use Moreau envelopes of convex functions, strong minimizers and strongly convex functions instead
 of subdifferentials.
 While Beer and Lucchetti obtained a similar result on generic well-posedness of convex optimization, their approach
 relies on epi-graphs of convex functions \cite{beerl1991, beer1992}.
 Our Moreau envelope approach is more accessible and natural
 to practical
 optimizers because taking the Moreau envelope is a popular regularization method used in the optimization community. We also
 give a systematic study of strongly convex functions, which is new to the best of our knowledge.
See also \cite{spingarn1979} for generic nature of constrained optimization problems,
and \cite{lucchetti2006} for well-posedness in optimization. For comprehensive generic results on fixed points
of firmly nonexpansive mappings and nonexpansive mappings, we refer the reader to \cite{reichzas2014}.

%%%%%
\section{Preliminaries}\label{sec:prelim}
%%%%%

\subsection{Notation}
All functions in this paper are defined on $\R^n,$ Euclidean space equipped with inner product $\langle x,y\rangle=\sum\limits_{i=1}^nx_iy_i,$ and induced norm $\|x\|=\sqrt{\langle x,x\rangle}.$ The extended real line $\R\cup\{\infty\}$ is denoted $\overline{\R}.$ We use $\dom f$ for the domain of $f,$ $\intt\dom f$ for the interior of the domain of $f,$ $\bdry\dom f$ for the boundary of the domain of $f,$ and $\epi f$ for the epigraph of $f.$
We use $\Gamma_0(X)$ to represent the set of proper lsc convex
functions on the space $X$ with the terms proper, lsc, and convex as defined in \cite{convmono, rockwets}. More precisely,
$f$ is proper if $-\infty\not\in f(X)$ and $\dom f\neq\varnothing$; $f$ is lsc at $x$ if
$x_{k}\rightarrow x$ implies $\liminf_{k\rightarrow}f(x_{k})
\geq f(x)$, when this is true at every $x\in X$ we call $f$ lsc on $X$; $f$ is convex if
$$(\forall x, y\in\dom f) (\forall 0\leq \alpha\leq 1)
\quad
f(\alpha x+(1-\alpha)y)\leq\alpha f(x)+(1-\alpha)f(y).$$
The symbol $G_\delta$ is used to indicate a generic set. The identity mapping or matrix is $\Id:\R^n\rightarrow\R^n: x\mapsto x.$
We use $\B_r(x)$ for the open ball centred at $x$ of radius $r,$ and $\B_r[x]$ for the closed ball.  For a set $C\subseteq\R^n$,
its closure is $\overline{C}$. The closed line segment between $x, y\in\R^n$  is $[x,y]:=
\{\lambda x+(1-\lambda)y: \ 0\leq \lambda \leq 1\}$.
We use $\overset{p}\rightarrow$ to indicate pointwise convergence, $\overset{e}\rightarrow$ for epi-convergence, and $\overset{u}\rightarrow$ for uniform convergence.
\subsection{Baire category}
Let $(X, d)$ be a metric space, where $X$ is a set and $d$ is a metric on $X$.
\begin{df}
A set $S\subseteq X$ is \emph{dense} in $X$ if every element of $X$ is either in $S,$ or a limit point of $S.$ A set is \emph{nowhere dense} in $X$ if the interior of its closure in $X$ is empty.
\end{df}
\begin{df}
A set $S\subseteq X$  is \emph{of first category (meagre)} if $S$ is a union of countably many nowhere dense sets.
A set $S\subseteq X$ is \emph{of second category (generic)} if $X\setminus S$ is of first category.
\end{df}

The following Baire category theorem is essential for this paper.
\begin{fact}
[Baire]\emph{(\cite[Theorem 1.47]{convanalgen} or \cite[Corollary 1.44]{convmono})}  Let $(X,d)$ be a complete metric space. Then any countable intersection of dense open subsets of $X$ is dense.
\end{fact}
\begin{fact}\label{separable}
Finite-dimensional space $\R^n$ is separable. That is, $\R^n$ has a countable subset that is dense in $\R^n.$
\end{fact}
\begin{proof}
This result is an extension of \cite[Example 1.3-7]{kreyszigfuncanal}, using the fact that the set of all $n$-tuples with rational components is a countable, dense subset of $\R^n.$
\end{proof}

%%%%%
\subsection{Convex analysis}
%%%%%

In this section we state several key facts about convex functions
that we need in order to prove the main results in subsequent sections.

\subsubsection{Subdifferentials of convex functions}
Let $f\in\Gamma_{0}(\R^n)$.
The set-valued mapping
   $$\partial f\colon \R^n\To \R^n\colon
   x\mapsto \menge{x^*\in \R^n}{(\forall y\in
\R^n)\; \langle y-x, x^*\rangle + f(x)\leq f(y)}$$ is the
{subdifferential
operator} of $f$.

 \begin{fact}\emph{\cite[Theorem 20.40]{convmono}}\label{subdiffmaxmono}
If $f\in\Gamma_0(\R^n),$ then $\partial f$ is maximally monotone.
\end{fact}
\begin{fact}\emph{(\cite[Theorem 12.41]{rockwets}, \cite[Theorem 2.51]{attouchwet86})}\label{maxmonoalmostconvex}
For any maximally monotone mapping $T:\R^n\rightrightarrows\R^n,$ the set $\dom T$ is almost convex. That is, there exists a convex set $C\subseteq \R^n$ such that $C\subseteq\dom T\subseteq\overline{C}.$ The same applies to the set $\ran T.$
\end{fact}
\begin{fact}\emph{\cite[Corollary 23.5.1]{convanalrock}}\label{inverse}
If $f\in\Gamma_0(\R^n),$ then $\partial f^*$ is the inverse of $\partial f$ in the sense of multivalued mappings, i.e. $x\in\partial f^*(x^*)$ if and only if $x^*\in\partial f(x).$
\end{fact}
\subsubsection{Convex functions and their Moreau envelopes}
\begin{df}
The \emph{Moreau envelope} of a proper, lsc function $f:\R^n\rightarrow\overline{\R}$ is defined as
$$e_\lambda f(x):=\inf\limits_y\left\{f(y)+\frac{1}{2\lambda}\|y-x\|^2\right\}.$$
The associated \emph{proximal mapping} is the (possibly empty) set of points at which this infimum is achieved, and is denoted $\Prox_f^\lambda:$
$$\Prox_f^\lambda(x):=\argmin\limits_y\left\{f(y)+\frac{1}{2\lambda}\|y-x\|^2\right\}.$$
\end{df}

In this paper, without loss of generality we use $\lambda=1.$ The theory developed here is equally applicable with any other choice of $\lambda>0.$

\begin{fact}\emph{(\cite[Proposition 12.29]{convmono} or \cite[Theorem 2.26]{rockwets})} Let $f\in\Gamma_{0}(\R^n)$. Then
$e_{1}f:\R^n\rightarrow\R$ is continuously differentiable on $\R^n$, and its gradient
$$\nabla e_{1}f=\Id-\Prox_{f}^{1}$$
is $1$-Lipschitz continuous, i.e., nonexpansive.
\end{fact}

One important concept for studying the convergence of extended-valued functions is epi-convergence, see, e.g., \cite{rockwets}.
\begin{df}
The \emph{lower epi-limit} of a sequence $\{f^\nu\}_{\nu\in\N}\subseteq\R^n$ is the function having as its epigraph the outer limit of the sequence of sets $\epi f^\nu:$
$$\epi(\eliminf_\nu f^\nu):=\limsup_\nu(\epi f^\nu).$$
Similarly, the \emph{upper epi-limit} of $\{f^\nu\}_{\nu\in\N}$ is the function having as its epigraph the inner limit of the sets $\epi f^\nu:$
$$\epi(\elimsup_\nu f^\nu):=\liminf_\nu(\epi f^\nu).$$
When these two functions coincide, the \emph{epi-limit} is said to exist and the functions are said to \emph{epi-converge} to $f:$
$$f^\nu\overset{e}\rightarrow f~~\mbox{if and only if }~~\epi f^\nu\rightarrow\epi f.$$
\end{df}

We refer the reader to \cite{rockwets, beer1992, vanderwerff} for further details on epi-convergence, e.g., continuity, stability and applications
in optimization.
The analysis of the limit properties of sequences of convex functions via their
Moreau envelopes is highlighted by the following fact.

\begin{fact}\emph{(\cite[Theorem 7.37]{rockwets}, \cite{attouch1984})}\label{fact:epi}
Let $\{f^\nu\}_{\nu\in\N}\subseteq\Gamma_0(\R^n),$ $f\in\Gamma_{0}(\R^n).$ Then
$$f^\nu\overset{e}\rightarrow f~~\mbox{ if and only if }~~e_1f^\nu\overset{p}\rightarrow e_1f.$$
Moreover, the pointwise convergence of $e_1f^\nu$ to $e_1f$ is uniform on all bounded subsets of $\R^n,$ hence yields epi-convergence to $e_1f$ as well.
\end{fact}

Two more nice properties about Moreau envelopes are:
\begin{fact}\emph{\cite[Example 1.46]{rockwets}}\label{fact1}
For any proper, lsc function $f:\R^n\rightarrow\overline{\R},$ $\inf f=\inf e_1f.$
\end{fact}

\begin{lem}\emph{\cite[Theorem 31.5]{convanalrock}}\label{lemmorconj}
Let $f\in\Gamma_0(\R^n).$ Then
$$e_1f(x)+e_1f^*(x)=\frac{1}{2}\|x\|^2.$$
\end{lem}

For more properties of Moreau envelopes of functions, we refer the reader to \cite{attouch1984, convmono, convanalrock, rockwets}.

%%%%%
\subsection{Strong minimizers, coercive convex functions and strongly convex functions}
%%%%%

We now present some basic properties of strong minimizers, strongly convex functions, and coercive functions.
\begin{df}
A function $f:\R^n\rightarrow \overline{\R}$ is said to attain a \emph{strong minimum} at $\bar{x}\in\R^n$
if
\begin{enumerate}
\item $f(\bar{x})\leq f(x)$ for all $x\in\dom f,$ and
\item $f(x_n)\rightarrow f(\bar{x})$ implies $x_n\rightarrow\bar{x}.$
\end{enumerate}
\end{df}
For further information on strong minimizers, we refer readers to \cite{lucchetti2006, borweinzhu, smoothvarprincip}.

\begin{df}
A function $f\in\Gamma_0(\R^n)$ is called \emph{coercive} if
$$\liminf\limits_{\|x\|\rightarrow\infty}\frac{f(x)}{\|x\|}=\infty.$$
\end{df}

\begin{df}
A function $f\in\Gamma_0(\R^n)$ is \emph{strongly convex} if there exists a modulus $\sigma>0$ such that $f-\frac{\sigma}{2}\|\cdot\|^2$ is convex. Equivalently, $f$ is strongly convex if there exists $\sigma>0$ such that for all $\lambda\in[0,1]$ and for all $x,y\in\R^n,$
$$f(\lambda x+(1-\lambda)y)\leq\lambda f(x)+(1-\lambda)f(y)-\frac{\sigma}{2}\lambda(1-\lambda)\|x-y\|^2.$$
\end{df}

\begin{df}
The \emph{Fenchel conjugate} of $f:\R^n\rightarrow\overline{\R}$ is defined as
$$f^*(v):=\sup\limits_x\{\langle v,x\rangle-f(x)\}.$$
\end{df}

\begin{fact}\emph{(\cite[Exercise 21 p. 83]{convanal}, \cite[Theorem 11.8]{rockwets})}\label{domfcoercive}
Let $f\in\Gamma_0(\R^n).$ Then $f$ is coercive if and only if $\dom f^*=\R^n.$
\end{fact}

\begin{lem}\label{lem3}
The function $f\in\Gamma_0(\R^n)$ is strongly convex if and only if $e_1f$ is strongly convex.
\end{lem}
\begin{proof} By \cite[Proposition 12.6]{rockwets}, $f$ is strongly convex if and only if $\nabla f^*$ is $\frac{1}{\sigma}$-Lipschitz for some $\sigma>0.$ Now
\begin{align*}
(e_1f)^*&=f^*+\frac{1}{2}\|\cdot\|^2,\mbox{ and}\\
\nabla(e_1f)^*&=\nabla f^*+\Id.
\end{align*}
Suppose that $f$ is strongly convex. Since $\nabla f^*$ is $\frac{1}{\sigma}$-Lipschitz, we have that $\nabla f^*+\Id$ is $\left(1+\frac{1}{\sigma}\right)$-Lipschitz. Hence, $\nabla(e_1f)^*$ is $\left(1+\frac{1}{\sigma}\right)$-Lipschitz. Then $e_1f$ is strongly convex, and we have proved one direction of the lemma. Working backwards with the same argument, the other direction is proved as well.
\end{proof}
\begin{lem}
Let $f\in\Gamma_0(\R^n).$ Then $f$ is coercive if and only if $e_1f$ is coercive.
\end{lem}
\begin{proof}
Suppose that $f$ is coercive. By Fact~\ref{domfcoercive}, a function is coercive if and only if its Fenchel conjugate is full-domain. Since $(e_1f)^*=f^*+\frac{1}{2}\|\cdot\|^2,$ we have that $(e_1f)^*$ is full-domain. Hence, $e_1f$ is coercive. To prove the other direction, suppose that $e_1f$ is coercive, and an identical argument shows that $f$ is coercive as well.
\end{proof}
\begin{lem}\label{lem4}
Let $f\in\Gamma_0(\R^n)$ be strongly convex. Then $f$ is coercive.
\end{lem}
\begin{proof} Since $f$ is strongly convex, $f$ can be written as $g+\frac{\sigma}{2}\|\cdot\|^2$ for some $g\in
\Gamma_{0}(\R^n)$ and $\sigma>0.$ Since $g$ is convex, $g$ is bounded below by a hyperplane. That is, there exist $\tilde{x}\in\R^n$ and $r\in\R$ such that
$$g(x)\geq\langle\tilde{x},x\rangle+r\mbox{ for all }x\in\R^n.$$
Hence,
$$f(x)\geq\langle\tilde{x},x\rangle+r+\frac{\sigma}{2}\|x\|^2\mbox{ for all }x\in\R^n.$$
This gives us that
$$\liminf\limits_{\|x\|\rightarrow\infty}\frac{f(x)}{\|x\|}=\infty.$$
\end{proof}

Note that a convex function can be coercive, but fail to be strongly convex. Consider the following example.
\begin{ex}
For $x\in\R,$ define
$$f(x):=\begin{cases}
(x+1)^2& \text{ if $x<-1$},\\
0  &\text{ if $-1\leq x\leq1$},\\
(x-1)^2 & \text{ if $x>1.$}
\end{cases}$$
Then $f(x)$ is coercive, but not strongly convex.
\end{ex}

\begin{proof}
It is elementary to show that $f$ is convex and coercive.
\begin{figure}[H]
\begin{center}\includegraphics[scale=0.3]{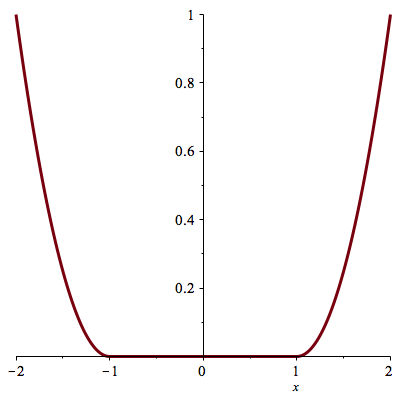}\end{center}
\end{figure}
Suppose that $f$ is strongly convex, and let $x=-1,$ $y=1,$ $\lambda=\frac{1}{2}.$ Then, for some $\sigma>0,$ we have
\begin{align*}
f(\lambda x+(1-\lambda)y)&\leq\lambda f(x)+(1-\lambda)f(y)-\frac{\sigma}{2}\lambda(1-\lambda)|x-y|^2,\\
f\left(\frac{1}{2}(-1)+\frac{1}{2}(1)\right)&\leq\frac{1}{2}f(-1)+\frac{1}{2}f(1)-\frac{\sigma}{2}\frac{1}{4}|-1-1|^2,\\
0&\leq-\frac{\sigma}{2},
\end{align*}
a contradiction. Therefore, $f$ is not strongly convex.
\end{proof}
\begin{lem}\label{lem5}
Let $f:\Gamma_0(\R^n)\rightarrow\overline{\R}$ be strongly convex. Then the (unique) minimizer of $f$ is a strong minimizer.
\end{lem}
\begin{proof} Let $f(x_k)\rightarrow\inf\limits_xf(x).$ Since $f$ is coercive by Lemma~\ref{lem4},
$\{x_k\}_{k=1}^\infty$ is bounded. By the Bolzano-Weierstrass Theorem, $\{x_k\}_{k=1}^\infty$ has a convergent subsequence $x_{k_j}\rightarrow\bar{x}.$ Since $f$ is lsc, we have that $\liminf\limits_{k\rightarrow\infty}f(x_k)\geq f(\bar{x}).$ Hence,
$$\inf\limits_xf(x)\leq f(\bar{x})\leq\inf\limits_xf(x).$$
Therefore, $f(\bar{x})=\inf\limits_xf(x).$ Since strong convexity implies strict convexity, $\argmin f(x)=\{\bar{x}\}$ is unique. As every subsequence of $\{x_k\}_{k=1}^\infty$ converges to the same limit $\bar{x},$ we conclude that $x_k\rightarrow\bar{x}.$
\end{proof}

To conclude this section, we provide an example that demonstrates the existence of functions that have strong minimizers, and yet are not strongly convex.
\begin{ex}
Let $f:\R\rightarrow\R,$ $f(x)=x^4.$ The function $f$ attains a strong minimum at $\bar{x}=0,$ but is not strongly convex.
\end{ex}
\begin{proof} By definition, $f$ is strongly convex if and only if there exists $\sigma>0$ such that $g(x):=x^4-\frac{\sigma}{2}x^2$ is convex. Since $g$ is a differentiable, univariable function, we know it is convex if and only if its second derivative is nonnegative for all $x\in\R.$ Since $g''(x)=12x^2-\sigma$ is clearly not nonnegative for any fixed $\sigma>0$ and all $x\in\R,$ we have that $g$ is not convex. Therefore, $f$ is not strongly convex.
Clearly zero is the minimum and minimizer of $f.$ Let $\{x_n\}_{n=1}^\infty\subseteq\R$ be such that $f(x_n)\rightarrow f(0)=0.$ Then
$
\lim\limits_{n\rightarrow\infty}x_n^4=0$ implies
$\lim\limits_{n\rightarrow\infty}x_n=0.$
Therefore, $f$ attains a strong minimum.
\end{proof}

%%%%%
\section{A complete metric space using Moreau envelopes}\label{sec:metric}
%%%%%
The principal tool we use is the Baire category theorem. To this end, we need a Baire space.
In this section, we establish a complete metric space whose distance function makes use of the Moreau envelope.
This metric has been used by Attouch-Wets in \cite[page 38]{attouchwet86}.
The distances used in the next section refer to the metric established here.

We begin with some properties on the Moreau envelope set
$$e_{1}(\Gamma_{0}(\R^n)):=\{e_{1}f:\ f\in\Gamma_{0}(\R^n)\}.$$

\begin{thm}\label{thm1}
The set $e_1(\Gamma_0(\R^n))$ is a convex set in $\Gamma_0(\R^n).$
\end{thm}
\begin{proof} Let $f_1,f_2\in\Gamma_0(\R^n),$ $\lambda\in[0,1].$ Then $e_1f_1,e_1f_2\in e_1(\Gamma_0(\R^n)).$ We need to show that $\lambda e_1f_1+(1-\lambda)e_1f_2\in e_1(\Gamma_0(\R^n)).$ By \cite[Theorem 6.2]{proxbas} with $\mu=1$ and $n=2,$ we have that $\lambda e_1f_1+(1-\lambda)e_1f_2$ is the Moreau envelope of the proximal average function $P_1(f,\lambda).$ By \cite[Corollary 5.2]{proxbas}, we have that $P_1(f,\lambda)\in\Gamma_0(\R^n).$ Hence, $e_1P_1(f,\lambda)\in e_1(\Gamma_0(\R^n)),$ and we conclude that $e_1(\Gamma_0(\R^n))$ is a convex set.
\end{proof}

On $e_{1}(\Gamma_{0}(\R^n)),$ define a metric by
\begin{equation}\label{e:m:envelope}
\tilde{d}(\tilde{f},\tilde{g}):= \sum\limits_{i=1}^\infty\frac{1}{2^i}\frac{\|\tilde{f}-\tilde{g}\|_i}{1+
\|\tilde{f}-\tilde{g}\|_i},
\end{equation}
where $\|\tilde{f}-\tilde{g}\|_i:=\sup\limits_{\|x\|\leq i}|\tilde{f}(x)-\tilde{g}(x)|$ and
$\tilde{f}, \tilde{g}\in e_{1}(\Gamma_{0}(\R^n))$.

Note that a sequence of functions in $(e_{1}(\Gamma_{0}(\R^n)),\tilde{d})$ converges if and only if the sequence
 converges uniformly on bounded sets, if and only if the sequence converges pointwise on $\R^n$.

\begin{thm}\label{thm2}
The metric space $(e_1(\Gamma_0(\R^n)),\tilde{d})$ is complete.
\end{thm}
\begin{proof} Let $\{f_k\}_{k=1}^\infty\subseteq\Gamma_0(\R^n),$ $f_k\rightarrow h.$ Then $e_1f_k\overset{p}\rightarrow g$ for some function $g.$ Our objective is to prove that $g$ is in fact the Moreau envelope of a proper, lsc, convex function. Since $f_k\in\Gamma_0(\R^n)$ for each $k,$ by Theorem \ref{thm1} $e_1f_k\in\Gamma_0(\R^n)$ for each $k.$ Then by \cite[Theorem 7.17]{rockwets}, we have that $e_1f_k\overset{e}\rightarrow g,$ and $e_1f_k\overset{u}\rightarrow g$ on bounded sets. Since $e_1f_k$ is convex and full-domain for each $k,$ $g$ is also convex and full-domain. By \cite[Theorem 11.34]{rockwets}, we have that $(e_1f_k)^*\overset{e}\rightarrow g^*,$ that is, $f_k^*+\frac{1}{2}\|\cdot\|^2\overset{e}\rightarrow g^*.$ Defining $h^*:=g^*-\frac{1}{2}\|\cdot\|^2,$ we have $f_k^*\overset{e}\rightarrow h^*\in\Gamma_0(\R^n).$ Then applying \cite[Theorem 11.34]{rockwets} again, we obtain $f_k\overset{e}\rightarrow h.$ Finally, using \cite[Theorem 7.37]{rockwets} we see that $e_1f_k\overset{e}\rightarrow e_1h,$ and we conclude that $g=e_1h\in\Gamma_0(\R^n).$ By Fact \ref{fact:epi}, we have pointwise and uniform convergence as well. Therefore, $e_1(\Gamma_0(\R^n))$ is closed under pointwise convergence topology.
\end{proof}

On $\Gamma_{0}(\R^n),$ we will use:
\begin{df}[Attouch-Wets metric]\label{defd}
For $f,g\in\Gamma_0(\R^n),$ define the distance function $d:$
$$d(f,g):=\sum\limits_{i=1}^\infty\frac{1}{2^i}\frac{\|e_1f-e_1g\|_i}{1+\|e_1f-e_1g\|_i}.$$
\end{df}

 In order to prove completeness of the space, we state the following lemma, whose simple proof is omitted.
\begin{lem}\label{lem6}
Define $a:[0,\infty)\rightarrow\R,$ $a(t):=\frac{t}{1+t}.$ Then
\begin{itemize}
\item[a)] $a$ is an increasing function, and
\item[b)] $t_1,t_2\geq0$ implies that $a(t_1+t_2)\leq a(t_1)+a(t_2).$
\end{itemize}
\end{lem}

\begin{prop}\label{prop1}
The space $(\Gamma_0(\R^n),d)$ where $d$ is the metric defined in Definition \ref{defd}, is a complete metric space.
\end{prop}
\begin{proof} Items M1-M4 show that $(\Gamma_0(\R^n),d)$ is a metric space, and item C shows that it is complete.\\
M1: Since
$$\sum\limits_{i=1}^\infty\frac{1}{2^i}=1,\mbox{ and }0\leq\frac{\|e_1f-e_1g\|_i}{1+\|e_1f-e_1g\|_i}<1\mbox{ for all }i,$$
we have that
$$\frac{1}{2^i}\geq\frac{1}{2^i}\frac{\|e_1f-e_1g\|_i}{1+\|e_1f-e_1g\|_i}\mbox{ for all }i.$$
Then
$$0\leq d(f,g)\leq1\mbox{ for all }f,g\in\Gamma_0(\R^n).$$
Hence, $d$ is real-valued, finite, and non-negative.\\
M2: We have
\begin{align*}
d(f,g)=0&\Leftrightarrow\sum\limits_{i=1}^\infty\frac{1}{2^i}\frac{\|e_1f-e_1g\|_i}{1+\|e_1f-e_1g\|_i}=0,\\
&\Leftrightarrow\|e_1f-e_1g\|_i=0\mbox{ for all }i,\\
&\Leftrightarrow e_1f(x)-e_1g(x)=0\mbox{ for all }x,\\
&\Leftrightarrow e_1f=e_1g,\\
&\Leftrightarrow f=g\mbox{ \cite[Corollary 3.36]{rockwets}.}
\end{align*}
Hence, $d(f,g)=0$ if and only if $f=g.$\\
M3: The fact that $d(f,g)=d(g,f)$ is trivial.\\
M4: By the triangle inequality,
$$\|e_1f-e_1g\|_i\leq\|e_1f-e_1h\|_i+\|e_1h-e_1g\|_i\mbox{ for all }f,g,h\in\Gamma_0(\R^n).$$
By applying Lemma \ref{lem6} (a), we have
$$\frac{\|e_1f-e_1g\|_i}{1+\|e_1f-e_1g\|_i}\leq\frac{\|e_1f-e_1h\|_i+\|e_1h-e_1g\|_i}{1+\|e_1f-e_1h\|_i+\|e_1h-e_1g\|_i}.$$
Then we apply Lemma \ref{lem6} (b) with $t_1=\|e_1f-e_1h\|_i$ and $t_2=\|e_1h-e_1g\|_i,$ and we have
$$\frac{\|e_1f-e_1g\|_i}{1+\|e_1f-e_1g\|_i}\leq\frac{\|e_1f-e_1h\|_i}{1+\|e_1f-e_1h\|_i}+\frac{\|e_1h-e_1g\|_i}{1+\|e_1h-e_1g\|_i}.$$
Multiplying both sides by $\frac{1}{2^i}$ and taking the summation over $i,$ we obtain the distance functions, which yields $d(f,g)\leq d(f,h)+d(h,g)$ for all $f,g,h\in\Gamma_0(\R^n).$
\item[C:] Let $\{f_k\}_{k=1}^\infty$ be a Cauchy sequence in $(\Gamma_0(\R^n),d),$ with $f_k\rightarrow h.$ Then for each $\varepsilon>0$ there exists $N_\varepsilon\in\N$ such that $d(f_j,f_k)<\varepsilon$ for all $j,k\geq N_\varepsilon.$ Fix $\varepsilon>0.$ Then there exists $N\in\N$ such that
$$\sum\limits_{i=1}^\infty\frac{1}{2^i}\frac{\|e_1f_j-e_1f_k\|_i}{1+\|e_1f_j-e_1f_k\|_i}<\varepsilon\mbox{ for all }j,k\geq N.$$
Then for any $i\in\N$ fixed, we have $\frac{\|e_1f_j-e_1f_k\|_i}{1+\|e_1f_j-e_1f_k\|_i}<2^{i}\varepsilon,$ so that $\|e_1f_j-e_1f_k\|_i<\frac{2^i\varepsilon}{1-2^i\varepsilon}=:\hat{\varepsilon}>0,$ for all $j,k\geq N.$ Notice that $\hat{\varepsilon}\searrow0$ as $\varepsilon\searrow0.$ This gives us that $\{e_1f_k\}_{k=1}^\infty$ is a Cauchy sequence on $B_i(x)$ for each $i\in\N,$ so that $e_1f_k\overset{p}\rightarrow g$ for some function $g.$
By the same arguments as in the proof of Theorem \ref{thm2}, we know that $g=e_1h\in\Gamma_0(\R^n),$ and hence $h\in\Gamma_0(\R^n).$ Therefore, $(\Gamma_0(\R^n),d)$ is closed, and is a complete metric space.
\end{proof}

On the set of Fenchel conjugates
$$(\Gamma_0(\R^n))^*:=\{f^*:\ f\in\Gamma_{0}(\R^n)\}$$ define a metric by
$\hat{d}(f,g):=d(f,g)$. Observe that $\Gamma_{0}(\R^n)=(\Gamma_{0}(\R^n))^*$.
\begin{cor}\label{thmmorconj}
Consider two metric spaces
$(\Gamma_0(\R^n),d)$ and $((\Gamma_0(\R^n))^*,\hat{d})$.
Define $$T:(\Gamma_0(\R^n),d)\rightarrow ((\Gamma_0(\R^n))^*,\hat{d}):
f\mapsto f^*.$$
 Then $T$ is a bijective isometry. Consequently,
$(\Gamma_0(\R^n),d)$ and $((\Gamma_0(\R^n))^*,\hat{d})$ are isometric.
\end{cor}
\begin{proof} Clearly $T$ is onto. Also, $T$ is injective because of the Fenchel-Moreau Theorem \cite[Theorem 13.32]{convmono}
 or \cite[Corollary 12.2.1]{convanalrock}. To see this,
let $Tf=Tg$. Then $f^*=g^*$, so $f=(f^*)^*=(g^*)^*=g$.
It remains to show that $T$ is an isometry: $(\forall f, g \in \Gamma_0(\R^n))$ $d(f,g)=d(f^*,g^*)=\hat{d}(Tf, Tg).$
Lemma \ref{lemmorconj} states that $e_1f+e_1f^*=\frac{1}{2}\|\cdot\|^2.$ Using this, we have
\begin{align*}
d(f^*,g^*)&=\sum\limits_{i=1}^\infty\frac{1}{2^i}\frac{\sup\limits_{\|x\|\leq i}|e_1f^*(x)-e_1g^*(x)|}{1+\sup\limits_{\|x\|\leq i}|e_1f^*(x)-e_1g^*(x)|}\\
&=\sum\limits_{i=1}^\infty\frac{1}{2^i}\frac{\sup\limits_{\|x\|\leq i}\left|\frac{1}{2}\|x\|^2-e_1f(x)-\frac{1}{2}\|x\|^2+e_1g(x)\right|}{1+\sup\limits_{\|x\|\leq i}\left|\frac{1}{2}\|x\|^2-e_1f(x)-\frac{1}{2}\|x\|^2+e_1g(x)\right|}\\
&=\sum\limits_{i=1}^\infty\frac{1}{2^i}\frac{\sup\limits_{\|x\|\leq i}|e_1g(x)-e_1f(x)|}{1+\sup\limits_{\|x\|\leq i}|e_1g(x)-e_1f(x)|}\\
&=d(f,g).
\end{align*}
\end{proof}

By Theorem~\ref{thm2}, $(e_{1}(\Gamma_{0}(\R^n)),\tilde{d})$ is a complete metric space.

\begin{cor}\label{c:convex:moreau} Consider two metric spaces $(\Gamma_{0}(\R^n), d)$ and $(e_{1}(\Gamma_{0}(\R^n)),\tilde{d})$.
Define $$T:\Gamma_{0}(\R^n)\rightarrow e_{1}(\Gamma_{0}(\R^n)):
f\mapsto e_{1}f.$$
 Then $T$ is a bijective isometry, so $(\Gamma_{0}(\R^n), d)$ and $(e_{1}(\Gamma_{0}(\R^n)),\tilde{d})$
are isometric.
\end{cor}

%%%%%
\section{Baire category results}\label{sec:main}
%%%%%

 This section is devoted to the main work of this paper. Ultimately, we show that the set of strongly convex functions is a meagre (Baire category one) set, whiel the set of convex functions that attain a strong minimum is a generic (Baire category two) set.

 %%%%%
\subsection{Characterizations of the strong minimizer}
%%%%%

The first proposition describes the relationship between a function and its Moreau envelope, pertaining to the strong minimum. Several more results regarding strong minima follow.

\begin{prop}\label{thm5}
Let $f:\R^n\rightarrow\RX.$ Then $f$ attains a strong minimum at $\bar{x}$ if and only if $e_1f$ attains a strong minimum at $\bar{x}.$
\end{prop}
\begin{proof} $(\Rightarrow)$ Assume that $f$ attains a strong minimum at $\bar{x}.$ Then $$\min\limits_xf(x)=\min\limits_xe_1f(x)=f(\bar{x})=e_1f(\bar{x}).$$ Let $\{x_k\}$ be such that $e_1f(x_k)\rightarrow e_1f(\bar{x}).$ We need to show that $x_k\rightarrow\bar{x}.$ Since $$e_1f(x_k)=f(v_k)+\frac{1}{2}\|v_k-x_k\|^2$$ for some $v_k,$ and $f(v_k)\geq f(\bar{x}),$ we have
\begin{equation}\label{vk}
0\leq\frac{1}{2}\|x_k-v_k\|^2+f(v_k)-f(\bar{x})=e_1f(x_k)-e_1f(\bar{x})\rightarrow0.
\end{equation}
Since both $\frac{1}{2}\|x_k-v_k\|^2\geq0$ and $f(v_k)-f(\bar{x})\geq0,$ equation \eqref{vk} tells us that $x_k-v_{k}\rightarrow 0$ and $f(v_k)\rightarrow f(\bar{x}).$ Since $\bar{x}$ is the strong minimizer of $f,$ we have $v_k\rightarrow\bar{x}.$ Therefore, $x_k\rightarrow\bar{x},$ and $e_1f$ attains a strong minimum at $\bar{x}.$\\
$(\Leftarrow)$ Assume that $e_1f$ attains a strong minimum at $\bar{x},$ $e_1f(\bar{x})=\min e_1f.$ Then $e_1f(x_k)\rightarrow e_1f(\bar{x})$ implies that $x_k\rightarrow\bar{x}.$ Let $f(x_k)\rightarrow f(\bar{x}).$ We have
$$f(\bar{x})\leq e_1f(\bar{x})\leq e_1f(x_k)\leq f(x_k).$$
Since $f(x_k)\rightarrow f(\bar{x}),$ we obtain
$$e_1f(x_k)\rightarrow f(\bar{x})=e_1f(\bar{x}).$$
Therefore, $x_k\rightarrow\bar{x},$ and $f$ attains a strong minimum at $\bar{x}.$
\end{proof}

\begin{thm}\label{thm7}
Let $f:\R^n\rightarrow\overline{\R}$ have a strong minimizer $\bar{x}.$ Then for all $m\in\N,$
$$\inf\limits_{\|x-\bar{x}\|\geq\frac{1}{m}}f(x)>f(\bar{x}).$$
\end{thm}
\begin{proof} Suppose that there exists $m\in\N$ such that $\inf\limits_{\|x-\bar{x}\|\geq\frac{1}{m}}f(x)=f(\bar{x}).$ Then there exists a sequence $\{x_k\}_{k=1}^\infty$ with $\|x_k-\bar{x}\|\geq\frac{1}{m}$ and $\lim\limits_{k\rightarrow\infty}f(x_k)=f(\bar{x}).$ Since $\bar{x}$ is the strong minimizer of $f,$
we have $x_k\rightarrow\bar{x},$ a contradiction.
\end{proof}
\begin{cor}\label{cor2}
Let $f:\R^n\rightarrow\overline{\R}$ have a strong minimizer $\bar{x}.$ Then for all $m\in\N,$
$$\inf\limits_{\|x-\bar{x}\|\geq\frac{1}{m}}e_1f(x)>e_1f(\bar{x}).$$
\end{cor}
\begin{proof} Applying Proposition \ref{thm5}, the proof is the same as that of Theorem \ref{thm7} replacing $f$ with $e_1f.$
\end{proof}

The next result describes a distinguished property of convex functions.
\begin{thm}\label{thm7.5}
Let $f\in\Gamma_0(\R^n).$ Then $f$ has a strong minimizer if and only if $f$ has a unique minimizer.
\end{thm}
\begin{proof} $(\Rightarrow)$ By definition, if $f$ has a strong minimizer, then that minimizer is unique.\\
$(\Leftarrow)$ Suppose $f$ has a unique minimizer $\bar{x}.$ Because
$f\in\Gamma_{0}(\R^n)$, by \cite[Theorem 8.7]{convanalrock},
all level-sets $\{x:f(x)\leq\alpha\},$ for any $\alpha\geq f(\bar{x}),$ have the same recession cone.
Since the recession cone of $\{x:f(x)\leq f(\bar{x})\}=\{\bar{x}\}$ is 0, \cite[Proposition 1.1.5]{convanal} gives us that
$$\liminf\limits_{\|x\|\rightarrow\infty}\frac{f(x)}{\|x\|}>0.$$
This, coupled with the fact that $f$ is convex, gives us that $f$ is coercive. Since $f$ is coercive and has a unique minimizer, we have that $\bar{x}$ is in fact a strong minimizer.
\end{proof}

\begin{ex}
The above property can fail when the function is nonconvex. Consider the continuous but nonconvex function $f:\R\rightarrow\R,$ $f(x)=\frac{x^2}{(x^4+1)}.$
\begin{figure}[H]
\begin{center}\includegraphics[scale=0.5]{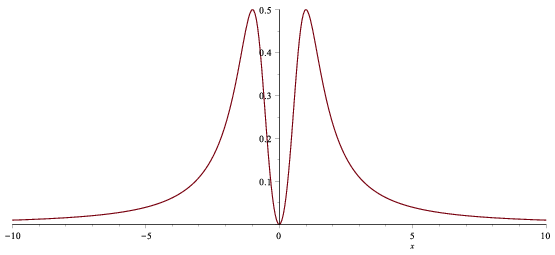}\end{center}
\end{figure}
\noindent The function has a unique minimizer $\bar{x}=0,$ but the minimizer is not strong, as any sequence $\{x^k\}$ that tends to $\pm\infty$ gives a sequence of function values that tends to $f(\bar{x}).$
\end{ex}

Using Theorem~\ref{thm7} and Corollary~\ref{cor2}, we can now single out two sets in $\Gamma_{0}(\R^n)$
which are very important for our later proofs.

\begin{df}
For any $m\in\N,$ define the sets $U_m$ and $E_m$ as follows:
\begin{align*}
U_m&:=\left\{f\in\Gamma_0(\R^n):\mbox{ there exists }z\in\R^n\mbox{ such that }\inf\limits_{\|x-z\|\geq\frac{1}{m}}f(x)-f(z)>0\right\},\\
E_m&:=\left\{f\in\Gamma_0(\R^n):\mbox{ there exists }z\in\R^n\mbox{ such that }\inf\limits_{\|x-z\|\geq\frac{1}{m}}e_1f(x)-e_1f(z)>0\right\}.
\end{align*}
\end{df}

\begin{prop}\label{thm6}
Let $f\in\bigcap\limits_{m\in\N}U_m.$ Then $f$ attains a strong minimum on $\R^n.$
\end{prop}
\begin{proof} The proof follows the method of \cite[Theorem II.1]{smoothvarprincip}. Since $f\in\bigcap\limits_{m\in\N}U_m,$ we have that for each $m\in\N$ there exists $z_m\in\R^n$ such that
$$f(z_m)<\inf\limits_{\|x-z_m\|\geq\frac{1}{m}}f(x).$$
Suppose that $\|z_p-z_m\|\geq\frac{1}{m}$ for some $p>m.$  By the definition of $z_m,$ we have
\begin{equation}\label{zm}
f(z_p)>f(z_m).
\end{equation}
Since $\|z_m-z_p\|\geq\frac{1}{m}>\frac{1}{p},$ we have
$$f(z_m)>f(z_p)$$
by the definition of $z_p.$ This contradicts equation \eqref{zm}. Thus, $\|z_p-z_m\|<\frac{1}{m}$ for each $p>m.$ This gives us that $\{z_m\}_{m=1}^\infty$ is a Cauchy sequence that converges to some $\bar{x}\in\R^n.$ It remains to be shown that $\bar{x}$ is the strong minimizer of $f.$ Since $f$ is lsc, we have
\begin{align*}
f(\bar{x})&\leq\liminf_{m\rightarrow\infty} f(z_m)\\
&\leq\liminf_{m\rightarrow\infty}
\left(\inf\limits_{\|x-z_m\|\geq\frac{1}{m}}f(x)\right)\\
&\leq\inf\limits_{x\in\R^n\setminus\{\bar{x}\}}f(x).
\end{align*}
Let $\{y_k\}_{k=1}^\infty\subseteq\R^n$ be such that $f(y_k)\rightarrow f(\bar{x}),$ and suppose that $y_k\not\rightarrow\bar{x}.$ Dropping to a subsequence if necessary, there exists $\varepsilon>0$ such that $\|y_k-\bar{x}\|\geq\varepsilon$ for all $k.$ Thus, there exists $p\in\N$ such that $\|y_k-z_p\|\geq\frac{1}{p}$ for all $k\in\N.$ Hence,
$$f(\bar{x})\leq f(z_p)<\inf\limits_{\|x-z_p\|\geq\frac{1}{p}}f(x)\leq f(y_k)$$
for all $k\in\N,$ a contradiction to the fact that $f(y_k)\rightarrow f(\bar{x}).$ Therefore, $\bar{x}$ is the strong minimizer of $f.$
\end{proof}
\begin{thm}\label{cor1}
Let $f\in\bigcap\limits_{m\in\N}E_m.$ Then $e_1f$ attains a strong minimum on $\R^n,$ so $f$ attains a strong minimum on $\R^n.$
\end{thm}
\begin{proof} Applying Proposition \ref{thm6}, for each $f\in\bigcap\limits_{m\in\N}E_m$, $e_{1}f$ has
a strong minimizer on $\R^n.$ Then
Proposition \ref{thm5} gives us that each corresponding $f$ has the same corresponding strong minimizer.
\end{proof}

%%%%%
\subsection{The set of strongly convex functions is dense, but of the first category}
%%%%%

Next, we turn our attention to the set of strongly convex functions. The objectives here are to show that the set is contained in both $U_m$ and $E_m,$ dense in $(\Gamma_0(\R^n),d),$ and meagre in $(\Gamma_0(\R^n),d).$
\begin{thm}\label{thm8}
Let $f:\R^n\rightarrow\overline{\R}$ be strongly convex. Then $f\in U_m$ and $f\in E_m$ for all $m\in\N.$
\end{thm}
\begin{proof} Since $f$ is strongly convex, $f$ has a unique minimizer $z.$ By Lemma \ref{lem5}, $z$ is a strong minimizer, so that for any sequence $\{x_k\}$ such that $f(x_k)\rightarrow f(\bar{x}),$ we must have $x_k\rightarrow\bar{x}.$ We want to show that
\begin{equation}\label{geqm}
\inf\limits_{\|x-z\|\geq\frac{1}{m}}f(x)-f(z)>0.
\end{equation}
For any $m\in\N,$ equation \eqref{geqm} is true by Theorem~\ref{thm7}.
 Therefore, $f\in U_m$ for all $m\in\N.$ By Lemma \ref{lem3}, $e_1f$ is strongly convex. Therefore, by the same reasoning as above, $f\in E_m$ for all $m\in\N.$
\end{proof}

We will need the following characterizations of strongly convex functions in later proofs. Note that
\ref{strong1}$\Rightarrow$\ref{strong3} has been done by Rockafellar \cite{monops}.
\begin{lem}\label{l:strongchar1}
 Let $f\in \Gamma_{0}(\R^n)$.
The following are equivalent:
\begin{enumerate}
\item\label{strong1} $f$ is strongly convex.
\item \label{strong2} $\Prox_{f}^{1}=k\Prox_{g}^{1}$ for some $0\leq k<1$ and $g\in\Gamma_{0}(\R^n)$.
\item\label{strong3} $\Prox_{f}^{1}=k N$ for some $0\leq k<1$ and $N:\R^n\rightarrow \R^n$ nonexpansive.
\end{enumerate}
\end{lem}
\begin{proof}
\ref{strong1}$\Rightarrow$\ref{strong2}: Assume that $f$ is strongly convex. Then
$f=g+\sigma q$ where $g\in\Gamma_{0}(\R^n)$, $q=\tfrac{1}{2}\|\cdot\|^2$, and $\sigma>0$. We have
\begin{align}
\Prox_{f}^{1} &=((1+\sigma)\Id+\partial g)^{-1}=\bigg((1+\sigma)\big(\Id+\frac{\partial g}{1+\sigma}\big)\bigg)^{-1}\\
&=\bigg(\Id+\frac{\partial g}{1+\sigma}\bigg)^{-1}\bigg(\frac{\Id}{1+\sigma}\bigg).
\end{align}
Define $\tilde{g}(x)=(1+\sigma)g(x/(1+\sigma))$. Then $\tilde{g}\in\Gamma_{0}(\R^n)$,
$\partial\tilde{g}=\partial g\circ \big(\frac{\Id}{1+\sigma}\big)$, so
\begin{align}
\Prox_{\tilde{g}}^{1} & =\bigg(\Id+\partial g\circ \bigg(\frac{\Id}{1+\sigma}\bigg)\bigg)^{-1}
=\bigg((1+\sigma)\bigg(\Id+\frac{\partial g}{1+\sigma}\bigg)\circ\bigg(\frac{\Id}{1+\sigma}\bigg)\bigg)^{-1}\\
&=(1+\sigma)\bigg(1+\frac{\partial g}{1+\sigma}\bigg)^{-1}\circ\bigg(\frac{\Id}{1+\sigma}\bigg)\\
&=(1+\sigma)\Prox_{f}^{1}.
\end{align}
Therefore, $\Prox_{f}^{1}=\tfrac{1}{1+\sigma}\Prox_{\tilde{g}}^{1}$.

\ref{strong2}$\Rightarrow$\ref{strong1}: Assume $\Prox_{f}^{1}=k\Prox_{g}^{1}$ for some $0\leq k<1$ and $g\in\Gamma_{0}(\R^n)$.
If $k=0$, then $f=\iota_{\{0\}}$, and $f$ is obviously strongly convex. Let us assume $0<k<1$.
The assumption
$(\Id+\partial f)^{-1}=k(\Id+\partial g)^{-1}$ gives
$\Id +\partial f=(\Id+\partial g)\circ (\Id/k)=\Id/k+\partial g \circ (\Id/k)$, so
$$\partial f=(1/k-1)\Id+\partial g(\Id/k).$$
Since $1/k>1$ and $\partial g\circ(\Id/k)$ is monotone, we have that $\partial f$ is strongly monotone,
which implies that $f$ is strongly convex.

\ref{strong2}$\Rightarrow$\ref{strong3}: This is clear because $\Prox_{g}^{1}$ is nonexpansive,
see, e.g., \cite[Proposition 12.27]{convmono}.

\ref{strong3}$\Rightarrow$\ref{strong2}: Assume $\Prox_{f}^{1}=kN$ where $0\leq k<1$ and $N$ is nonexpansive.
If $k=0$, then $\Prox_{f}^{1}=0=0\cdot 0$, so \ref{strong2} holds because $\Prox_{\iota_{\{0\}}}=0$.
If $0<k<1$, then
$N=1/k \Prox_{f}^{1}.$
As $$\Prox_{f}^{1}=(\Id+\partial f)^{-1}=\nabla (q+f)^{*}=\nabla e_{1}(f^*),$$
we have $N=\nabla (e_{1}(f^*)/k)$. This means that $N$ is nonexpansive and the gradient of a differentiable
convex function. By the Baillon-Haddad theorem \cite{baillonhaddad} or \cite[Corollary 18.16]{convmono},
$N=\Prox_{g}^{1}$ for some $g\in\Gamma_{0}(\R^n)$.
Therefore, $\Prox_{f}^{1}=k \Prox_{g}^{1}$, i.e., \ref{strong2} holds true.
\end{proof}

\begin{thm}\label{thm9}
The set of strongly convex functions is dense in $(\Gamma_0(\R^n),d).$ Equivalently,
the set of strongly convex functions is dense in $(e_{1}(\Gamma_{0}(\R^n)),\tilde{d})$.
\end{thm}
\begin{proof} Let $0<\varepsilon<1$ and $f\in\Gamma_0(\R^n).$ It will suffice to find $h\in\Gamma_0(\R^n)$ such that $h$ is strongly convex and $d(h,f)<\varepsilon.$ For $0<\sigma<1,$ define $g\in\Gamma_0(\R^n)$ by way of the proximal mapping:
$$\Prox_g^1:=(1-\sigma)\Prox_f^1=(1-\sigma)\Prox_{f}^{1}+\sigma \Prox_{\iota_{\{0\}}}.$$
Such a $g\in\Gamma_{0}(\R^n)$ does exists because $g$ is the proximal average of $f$ and $\iota_{\{0\}}$ by \cite{proxbas}, and
$g$ is strongly convex because of Lemma~\ref{l:strongchar1}.
Define $h\in\Gamma_{0}(\R^n)$ by
$$h:=g-e_1g(0)+e_1f(0).$$
Then $e_1h=e_1g-e_1g(0)+e_1f(0),$ so that
\begin{equation}\label{eq1}
e_1h(0)=e_1f(0),
\end{equation}
and $\Prox_h^1=\Prox_g^1.$ Fix $N$ large enough that $\sum\limits_{i=N}^\infty\frac{1}{2^i}<\frac{\varepsilon}{2}.$ Then
\begin{equation}\label{eq2}
\sum\limits_{i=N}^\infty\frac{1}{2^i}\frac{\|e_1f-e_1g\|_i}{1+\|e_1f-e_1g\|_i}\leq\sum\limits_{i=N}^\infty\frac{1}{2^i}<\frac{\varepsilon}{2}.
\end{equation}
Choose $\sigma$ such that
\begin{equation}\label{eq3}
0<\sigma<\frac{\varepsilon}{2-\varepsilon}\frac{1}{N(N+\|\Prox_f^1(0)\|))}.
\end{equation}
This gives us that
\begin{equation}\label{eq4}
\frac{\sigma N(N+\|P-1f(0)\|)}{1+\sigma N(N+\|P-1f(0)\|)}<\frac{\varepsilon}{2}.
\end{equation}
By equation \eqref{eq1} and the Mean Value Theorem, for some $c\in[x,0]$ we have
\begin{align*}
e_1h(x)-e_1f(x)&=e_1h(x)-e_1f(x)-(e_1h(0)-e_1f(0))\\
&=\langle\nabla e_1h(c)-\nabla e_1f(c),x-0\rangle\\
&=\langle(\Id-\Prox_h^1)(c)-(\Id-\Prox_f^1)(c),x-0\rangle\\
&=\langle-\Prox_h^1(c)+\Prox_f^1(c),x-0\rangle\\
&=\langle-(1-\sigma)\Prox_f^1(c)+\Prox_f^1(c),x\rangle\\
&=\langle\sigma\Prox_f^1(c),x\rangle.
\end{align*}
Using the triangle inequality, the Cauchy-Schwarz inequality, and the fact that $\Prox_f^1$ is nonexpansive, we obtain
\begin{align*}
|e_1h(x)-e_1f(x)|&\leq\sigma\|\Prox_f^1(c)\|\|x\|\\
&=\sigma\|\Prox_f^1(c)-\Prox_f^1(0)+\Prox_f^1(0)\|\|x\|\\
&\leq\sigma(\|\Prox_f^1(c)-\Prox_f^1(0)\|+\|\Prox_f^1(0)\|)\|x\|\\
&\leq\sigma(\|c\|+\|\Prox_f^1(0)\|)\|x\|\\
&\leq\sigma(\|x\|+\|\Prox_f^1(0)\|)\|x\|\\
&\leq\sigma N(N+\|\Prox_f^1(0)\|),
\end{align*}
when $\|x\|\leq N.$ Therefore, $\|e_1h-e_1f\|_N\leq\sigma N(N+\|\Prox_f^1(0)\|).$ Applying equation \eqref{eq4}, this implies that
\begin{equation}\label{eq5}
\frac{\|e_1f-e_1g\|_N}{1+\|e_1f-e_1g\|_N}\leq\frac{\sigma N(N+\|\Prox_f^1(0)\|)}{1+\sigma N(N+\|\Prox_f^1(0)\|)}<\frac{\varepsilon}{2}.
\end{equation}
Now considering the first $N-1$ terms of our $d$ function, we have
\begin{align}
\sum\limits_{i=1}^{N-1}\frac{1}{2^i}\frac{\|e_1f-e_1g\|_i}{1+\|e_1f-e_1g\|_i}&\leq\sum\limits_{i=1}^{N-1}\frac{1}{2^i}\frac{\|e_1f-e_1g\|_N}{1+\|e_1f-e_1g\|_N}\nonumber\\
&=\frac{\|e_1f-e_1g\|_N}{1+\|e_1f-e_1g\|_N}\sum\limits_{i=1}^{N-1}\frac{1}{2^i}\nonumber\\
&<\frac{\|e_1f-e_1g\|_N}{1+\|e_1f-e_1g\|_N}.\label{eq6}
\end{align}
When equation \eqref{eq3} holds, combining equations \eqref{eq2}, \eqref{eq5}, and \eqref{eq6} yields $d(h,f)<\varepsilon.$ Hence, for any arbitrary $f\in\Gamma_0(\R^n)$ and $0<\varepsilon<1,$ there exists a strongly convex function $h\in\Gamma_0(\R^n)$ such that $d(h,f)<\varepsilon.$ That is, the set of strongly convex functions is dense in $(\Gamma_0(\R^n),d).$
Because $(\Gamma_0(\R^n),d)$ and $(e_{1}(\Gamma_{0}(\R^n)),\tilde{d})$ are isometric by Corollary~\ref{c:convex:moreau},
it suffices to apply Lemma~\ref{lem3}. The proof is complete.
\end{proof}

\begin{thm}\label{strongconvmeagre}
The set of strongly convex functions is meagre in $(e_1(\Gamma_0(\R^n)),\tilde{d})$ where $\tilde{d}$ is
given by \eqref{e:m:envelope}. Equivalently, in $(\Gamma_{0}(\R^n),d)$ the set of strongly convex
function is meagre.
\end{thm}
\begin{proof}
Denote the set of strongly convex functions in $e_1(\Gamma_0(\R^n))$ by $S.$ Define
$$F_m:=\left\{g\in e_1(\Gamma_0(\R^n)):g-\frac{1}{2m}\|\cdot\|^2\mbox{ is convex on }\R^n\right\}.$$
We show that
\begin{itemize}
\item[a)] $S=\bigcup\limits_{m\in\N}F_m,$
\item[b)] for each $m\in\N,$ the set $F_m$ is closed in $e_1(\Gamma_0(\R^n)),$ and
\item[c)] for each $m\in\N,$ the set $F_m$ has empty interior.
\end{itemize}
Then $S$ will have been shown to be a countable union of closed, nowhere dense sets, hence first category.
\begin{itemize}
\item[a)] $(\Rightarrow)$ Let $f\in S.$ Then there exists $\sigma>0$ such that $f-\frac{\sigma}{2}\|\cdot\|^2$ is convex. Note that this means $f-\frac{\tilde{\sigma}}{2}\|\cdot\|^2$ is convex for all $\tilde{\sigma}\in(0,\sigma).$ Since $\sigma>0,$ there exists $m\in\N$ such that $0<\frac{1}{m}<\sigma.$ Hence, $f-\frac{1}{2m}\|\cdot\|^2$ is convex, and $f\in F_{m}.$ Therefore, $S\subseteq\bigcup\limits_{m\in\N}F_m.$\\
$(\Leftarrow)$ Let $f\in F_m$ for some $m\in\N.$ Then $f-\frac{1}{2m}\|\cdot\|^2$ is convex. Thus, with $\sigma=\frac{1}{m},$ we have that there exists $\sigma>0$ such that $f-\frac{\sigma}{2}\|\cdot\|^2$ is convex, which is the definition of strong convexity of $f.$ Therefore, $F_m\subseteq S,$ and since this is true for every $m\in\N,$ we have $\bigcup\limits_{m\in\N}F_m\subseteq S.$
\item[b)] Let $g\not\in F_m.$ Then $g-\frac{1}{2m}\|\cdot\|^2$ is not convex. Equivalently, there exist $\lambda\in(0,1)$ and $x,y\in\R^n$ such that
\begin{equation}\label{gnotconvex}
\frac{g(\lambda x+(1-\lambda)y)-\lambda g(x)-(1-\lambda)g(y)}{\lambda(1-\lambda)}>-\frac{\|x-y\|^2}{2m}.
\end{equation}
Let $N>\max\{\|x\|,\|y\|\}.$ Choose $\varepsilon>0$ such that when $\tilde{d}(f,g)<\varepsilon$ for $f\in e_1(\Gamma_0(\R^n)),$ we have $\|f-g\|_N<\tilde{\varepsilon}$ for some $\tilde{\varepsilon}>0.$ In particular,\footnotesize
\begin{align*}
\frac{f(\lambda x+(1-\lambda)y)-\lambda f(x)-(1-\lambda)f(y)}{\lambda(1-\lambda)}=&\frac{g(\lambda x+(1-\lambda)y)-\lambda g(x)-(1-\lambda)g(y)}{\lambda(1-\lambda)}\\
&+\frac{(f-g)(\lambda x+(1-\lambda)y)-\lambda (f-g)(x)-(1-\lambda)(f-g)(y)}{\lambda(1-\lambda)}\\
>&\frac{g(\lambda x+(1-\lambda)y)-\lambda g(x)-(1-\lambda)g(y)}{\lambda(1-\lambda)}-\frac{4\tilde{\varepsilon}}{\lambda(1-\lambda)}.
\end{align*}\normalsize
Hence, when $\tilde{\varepsilon}$ is sufficiently small, which can be achieved by making $\varepsilon$
sufficiently small, we have
$$\frac{f(\lambda x+(1-\lambda)y)-\lambda f(x)-(1-\lambda)f(y)}{\lambda(1-\lambda)}>-\frac{\|x-y\|^2}{2m}.$$
This gives us, by equation \eqref{gnotconvex}, that $f-\frac{1}{2m}\|\cdot\|^2$ is not convex. Thus, $f\not\in F_m,$ so $e_1(\Gamma_0(\R^n))\setminus F_m$ is open, and therefore $F_m$ is closed.
\item[c)] That $\intt F_m=\emptyset$ is equivalent to saying that $e_1(\Gamma_0(\R^n))\setminus F_m$ is dense. Thus, it suffices to show that for every $\varepsilon>0$ and every $g\in e_1(\Gamma_0(\R^n)),$ the open ball $\B_\varepsilon(g)$ contains an element of $e_1(\Gamma_0(\R^n))\setminus F_m.$\\
If $g\in e_1(\Gamma_0(\R^n))\setminus F_m,$ then there is nothing to prove. Assume that $g\in F_m.$ Then $g$ is $\frac{1}{2m}$-strongly convex, and has a strong minimizer $\bar{x}$ by Lemma \ref{lem5}. As $g\in e_1(\Gamma_0(\R^n)),$ $g=e_1f$ for some $f\in\Gamma_0(\R^n).$ We consider two cases.
\begin{itemize}
\item[Case 1:] Suppose that for every $\frac{1}{k}>0,$ there exists $x_k\neq\bar{x}$ such that $f(x_k)<f(\bar{x})+\frac{1}{k}.$ Define $h_k:=\max\left\{f,f(\bar{x})+\frac{1}{k}\right\}.$ Then
$$\min h_k=f(\bar{x})+\frac{1}{k},~f\leq h_k<f+\frac{1}{k},$$
so that $e_1f\leq e_1h_k\leq e_1f+\frac{1}{k}.$ We have $g_k:=e_1h_k\in e_1(\Gamma_0(\R^n)),$ and $\|g_k-g\|_i<\frac{1}{k}$ for all $i\in\N.$ Choosing $k$ sufficiently large guarantees that $\tilde{d}(g_k,g)<\varepsilon.$
We see that $g_k$ does not have a strong minimizer by noting that for every $k,$ $f(\bar{x})<f(\bar{x})+\frac{1}{k},$ $f(x_k)<f(\bar{x})+\frac{1}{k},$ and $h_k(\bar{x})=h_k(x_k)=f(\bar{x})+\frac{1}{k}.$ Thus, $h_k$ does not have a strong minimizer, which implies that $g_k=e_1h_k$ does not either, by Proposition \ref{thm5}. Therefore, $g_k\not\in F_m.$
\item[Case 2:] If Case 1 is not true, then there exists $k$ such that $f(x)\geq f(\bar{x})+\frac{1}{k}$ for every $x\neq\bar{x}.$ Then we claim that $f(x)=\infty$ for all $x\neq\bar{x}.$ Suppose for the purpose of contradiction that there exists $x\neq\bar{x}$ such that $f(x)<\infty.$ As $f\in\Gamma_0(\R^n),$ the function $\phi:[0,1]\rightarrow\R$ defined by $\phi(t):=f(tx+(1-t)\bar{x})$ is continuous by \cite[Proposition 2.1.6]{convanalgen}. This contradicts the assumption, therefore,
$$f(x)=\iota_{\{\bar{x}\}}(x)+f(\bar{x}).$$
Consequently,
$$g(x)=e_1f(x)=f(\bar{x})+\frac{1}{2}\|x-\bar{x}\|^2.$$
Now for every $j\in\N,$ define $f_j:\R^n\rightarrow\overline{\R},$
$$f_j(x):=\begin{cases}
f(\bar{x}),&\|x-\bar{x}\|\leq\frac{1}{j},\\\infty,&\mbox{otherwise.}
\end{cases}$$
We have $f_j\in\Gamma_0(\R^n),$ and
$$g_j(x):=e_1f_j(x)=\begin{cases}
f(\bar{x}),&\|x-\bar{x}\|\leq\frac{1}{j},\\
f(\bar{x})+\frac{1}{2}\left(\|x-\bar{x}\|-\frac{1}{j}\right)^2,&\|x-\bar{x}\|>\frac{1}{j}.
\end{cases}$$
Then $\{g_j(x)\}_{j\in\N}$ converges pointwise to $e_1f=g,$ by \cite[Theorem 7.37]{rockwets}. Thus, for sufficiently large $j,$ $\tilde{d}(g_j,g)<\varepsilon.$ Since $g_j$ is constant on $\B_{\frac{1}{j}}(\bar{x}),$ $g_j$ is not strongly convex, so $g_j\not\in F_m.$
\end{itemize}
\end{itemize}
Properties a), b) and c) all together show that the set of strongly convex function is meagre in $(e_{1}(\Gamma_{0}(\R^n),\tilde{d})$.
Note that $(e_{1}(\Gamma_{0}(\R^n),\tilde{d})$ and $(\Gamma_{0}(\R^n),d)$ are isometric by Corollary~\ref{c:convex:moreau}.
The proof is complete
by using Lemma~\ref{lem3}.
\end{proof}

%%%%%
\subsection{The set of convex functions with strong minimizers is of second category}
%%%%%

We present properties of the sets $U_m$ and $E_m,$ and show that the set of convex functions
that attain a strong minimum is a generic set in $(\Gamma_{0}(\R^n),d)$.
\begin{lem}\label{cor3}
The sets $U_m$ and $E_m$ are dense in $(\Gamma_0(\R^n),d).$
\end{lem}
\begin{proof} This is immediate by combining Theorems \ref{thm8} and \ref{thm9}.
\end{proof}

To continue, we need the following result, which holds in $\Gamma_{0}(X)$ where $X$ is any Banach space.
\begin{lem}\label{lem1}
Let $f\in\Gamma_0(\R^n),$ $m\in\N,$ and fix $z\in\dom f.$ Then
$$\inf\limits_{\|x-z\|\geq\frac{1}{m}}f(x)-f(z)>0\mbox{ if and only if }\inf\limits_{m\geq\|x-z\|\geq\frac{1}{m}}f(x)-f(z)>0.$$
\end{lem}
\begin{proof}
$(\Rightarrow)$ Suppose that for $z$ fixed, $\inf\limits_{\|x-z\|\geq\frac{1}{m}}f(x)-f(z)>0.$ Since
$$\inf\limits_{m\geq\|x-z\|\geq\frac{1}{m}}f(x)-f(z)
\geq\inf\limits_{\|x-z\|\geq\frac{1}{m}}f(x)-f(z)>0,$$
we have
$\inf\limits_{m\geq\|x-z\|\geq\frac{1}{m}}f(x)-f(z)>0.$

$(\Leftarrow)$ Let $\inf\limits_{m\geq\|x-z\|\geq\frac{1}{m}}f(x)-f(z)>0,$ and suppose that
$$\inf\limits_{\|x-z\|\geq\frac{1}{m}}f(x)-f(z)\leq0.$$
Then for each $\frac{1}{k}$ with $k\in\N,$ there exists $y_k$ with $\|y_k-z\|\geq\frac{1}{m}$ such that $f(y_k)\leq f(z)+\frac{1}{k}.$ Take $z_k\in[y_k,z]\cap\left\{x\in\R^n:\ m\geq\|x-z\|\geq\frac{1}{m}\right\}\neq\emptyset.$ Then
$$z_k=\lambda_ky_k+(1-\lambda_k)z$$
for some $\lambda_k\in [0,1].$ By the convexity of $f$, we have
\begin{align*}
f(z_k)&=f(\lambda_ky_k+(1-\lambda_k)z)\leq\lambda_kf(y_k)+(1-\lambda_k)f(z)\\
&\leq\lambda_kf(z)+(1-\lambda_k)f(z)+\frac{\lambda_k}{k}\\
&=f(z)+\frac{\lambda_k}{k}\leq f(z)+\frac{1}{k}.
\end{align*}
Now $\inf\limits_{m\geq\|x-z\|\geq\frac{1}{m}}f(x)\leq f(z_k)\leq f(z)+\frac{1}{k},$ so when $k\rightarrow\infty$ we obtain
$$\inf\limits_{m\geq\|x-z\|\geq\frac{1}{m}}f(x)-f(z)\leq0.$$
This contradicts the fact that $\inf\limits_{m\geq\|x-z\|\geq\frac{1}{m}}f(x)-f(z)>0.$ Therefore,
$\inf\limits_{\|x-z\|\geq\frac{1}{m}}f(x)-f(z)>0.$
\end{proof}

\begin{lem}\label{thm10}
The set $E_m$ is an open set in $(\Gamma_0(\R^n),d).$
\end{lem}
\begin{proof} Fix $m\in\N,$ and let $f\in E_m.$ Then there exists $z\in\R^n$ such that $\inf\limits_{\|x-z\|\geq\frac{1}{m}}e_1f(x)-e_1f(z)>0.$ Hence, by Lemma \ref{lem1},
$$\inf\limits_{m\geq\|x-z\|\geq\frac{1}{m}}e_1f(x)-e_1f(z)>0.$$
Choose $j$ large enough that $\B_m[z]\subseteq \B_j(0).$ Let $g\in\Gamma_0(\R^n)$ be such that $d(f,g)<\varepsilon,$
where
\begin{equation}\label{eq7}
0<\varepsilon<\frac{\inf\limits_{m\geq\|x-z\|\geq
\frac{1}{m}}e_1f(x)-e_1f(z)}{2^j\left(2+\inf\limits_{m\geq\|x-z\|\geq\frac{1}{m}}e_1f(x)-e_1f(z)\right)}<\frac{1}{2^{j}}.
\end{equation}
The reason for this bound on $\varepsilon$ will become apparent at the end of the proof. Then
$$\sum\limits_{i=1}^\infty\frac{1}{2^i}\frac{\|e_1f-e_1g\|_i}{1+\|e_1f-e_1g\|_i}<\varepsilon.$$
In particular for our choice of $j,$  we have that $2^{j}\varepsilon<1$ by \eqref{eq7}, and that
\begin{align*}
\frac{1}{2^j}\frac{\|e_1f-e_1g\|_j}{1+\|e_1f-e_1g\|_j}&<\varepsilon,\\
\|e_1f-e_1g\|_j&<2^j\varepsilon(1+\|e_1f-e_1g\|_j),\\
\sup\limits_{\|x\|\leq j}|e_1f(x)-e_1g(x)|(1-2^j\varepsilon)&<2^j\varepsilon,\\
\sup\limits_{\|x\|\leq j}|e_1f(x)-e_1g(x)|&<\frac{2^j\varepsilon}{1-2^j\varepsilon}.
\end{align*}
Define $\alpha:=\frac{2^j\varepsilon}{1-2^j\varepsilon}.$ Then $\sup\limits_{\|x\|\leq j}|e_1f(x)-e_1g(x)|<\alpha.$ Hence,
$$|e_1f(x)-e_1g(x)|<\alpha\mbox{ for all }x\mbox{ with }\|x\|\leq j.$$
In other words,
$$e_1f(x)-\alpha<e_1g(x)<e_1f(x)+\alpha\mbox{ for all }x\mbox{ with }\|x\|\leq j.$$
Since $\B_m[z]\subseteq \B_j(0),$ we can take the infimum over $m\geq\|x-z\|\geq\frac{1}{m}$ to obtain
\begin{equation}\label{eq8}
\inf\limits_{m\geq\|x-z\|\leq\frac{1}{m}}e_1f(x)-\alpha\leq
\inf\limits_{m\geq\|x-z\|\geq\frac{1}{m}}e_1g(x)\leq\inf\limits_{m\geq\|x-z\|\geq\frac{1}{m}}e_1f(x)+\alpha.
\end{equation}
Using equation \eqref{eq8} together with the fact that $|e_1g(z)-e_1f(z)|<\alpha$ yields
\begin{align*}
\inf\limits_{m\geq\|x-z\|\geq\frac{1}{m}}e_1g(x)-e_1g(z)&\geq\left(\inf\limits_{m\geq\|x-z\|\geq\frac{1}{m}}e_1f(x)-\alpha\right)-(e_1f(z)+\alpha)\\
&=\inf\limits_{m\geq\|x-z\|\geq\frac{1}{m}}e_1f(x)-e_1f(z)-2\alpha.
\end{align*}
Hence, if
\begin{equation}\label{eq9}
\alpha<\frac{\inf\limits_{m\geq\|x-z\|\geq\frac{1}{m}}e_1f(x)-e_1f(z)}{2},
\end{equation}
we have
\begin{equation}\label{eq10}
\inf\limits_{m\geq\|x-z\|\geq\frac{1}{m}}e_1g(x)-e_1g(z)>0.
\end{equation}
Recalling that $\alpha=\frac{2^j\varepsilon}{1-2^j\varepsilon},$ we solve equation \eqref{eq9} for $\varepsilon$ to obtain
$$\varepsilon<\frac{\inf\limits_{m\geq\|x-z\|\geq\frac{1}{m}}e_1f(x)-e_1f(z)}{2^j\left(2+\inf\limits_{m\geq\|x-z\|\geq\frac{1}{m}}e_1f(x)-e_1f(z)\right)}.$$
Thus, equation \eqref{eq10} is true whenever $d(f,g)<\varepsilon$ for any $\varepsilon$ that respects equation \eqref{eq7}. Applying Lemma \ref{lem1} to equation \eqref{eq10}, we conclude that
$$\inf\limits_{\|x-z\|\geq\frac{1}{m}}e_1g(x)-e_1g(z)>0.$$
Hence, if $g\in\Gamma_0(\R^n)$ is such that $d(f,g)<\varepsilon,$ then $g\in E_m.$ Therefore, $E_m$ is open.
\end{proof}

We are now ready to present the main results of the paper.
\begin{thm}\label{generictheorem}
In $X:=(\Gamma_0(\R^n),d),$ the set $S:=\{f\in\Gamma_0(\R^n):f\mbox{ attains a strong minimum}\}$ is generic.
\end{thm}
\begin{proof}
By Lemmas \ref{cor3} and \ref{thm10}, we have that $E_m$ is open and dense in $X.$ Hence,
$G:=\bigcap\limits_{m\in\N}E_m$ is a countable intersection of open, dense sets in $X$,
and as such $G$ is generic in $X.$ Let $f\in G.$ By Corollary \ref{cor1},
%$e_1f$ attains a strong minimum. Then by Proposition \ref{thm5},
$f$ attains a strong minimum on $\R^n.$ Thus, every element of $G$ attains a strong minimum on $\R^n.$ Since $G$ is generic in $X$ and $G\subseteq S,$ we conclude that $S$ is generic in $X.$
\end{proof}
\begin{thm}\label{t:fullrange}
In $X:=(\Gamma_0(\R^n),d),$ the set $S:=\{f\in\Gamma_0(\R^n):f\mbox{ is coercive}\}$ is generic.
\end{thm}
\begin{proof}
Define the set $\Gamma_1(\R^n):=\Gamma_0(\R^n)+x^*,$ in the sense that for any function $f\in\Gamma_0(\R^n),$ the function $f+\langle x^*,\cdot\rangle\in\Gamma_1(\R^n).$ Since any such $f+\langle x^*,\cdot\rangle$ is proper, lsc, and convex, we have $\Gamma_1(\R^n)\subseteq\Gamma_0(\R^n).$ Now, since for any $f\in\Gamma_0(\R^n)$ we have that $f-x^*\in\Gamma_0(\R^n),$ this gives us that $f\in\Gamma_0(\R^n)+x^*=\Gamma_1(\R^n).$ Therefore, $\Gamma_1(\R^n)=\Gamma_0(\R^n).$ By Theorem \ref{generictheorem}, there exists a generic set $G\subseteq\Gamma_0(\R^n)$ such that for every $f\in G,$ $f$ attains a strong minimum at some point $x,$ and hence $0\in\partial f(x).$ Then, given any $x^*$ fixed, there exists a generic set $G_{x^*}$ that contains a dense $G_\delta$ set, such that $0\in\partial(f+x^*)(x).$ Thus, for each $f\in G_{x^*}$ there exists $x\in\R^n$ such that $-x^*\in\partial f(x)$.  By Fact \ref{separable}, it is possible to construct the set $D:=\{-x_i^*\}_{i=1}^\infty$ such that $\overline{D}=\R^n.$ Then each set $G_{x_i^*},$ $i\in\N,$ contains a dense $G_\delta$ set. Therefore, the set $G:=\bigcap\limits_{i=1}^\infty G_{x_i^*}$ contains a dense $G_\delta$ set. Let $f\in G.$ Then for each $i\in\N,$ $-x_i^*\in\partial f(x)$ for some $x\in\R^n.$ That is, $-x_i^*\in\ran\partial f.$ So $D:=\bigcup\limits_{i=1}^\infty\{-x_i^*\}\subseteq\ran\partial f,$ and $\overline{D}\subseteq\overline{\ran\partial f}.$ Since $\overline{D}=\R^n,$ we have $\R^n=\overline{\ran\partial f}.$ By Facts \ref{subdiffmaxmono} and \ref{maxmonoalmostconvex}, $\ran\partial f$ is almost convex; there exists a convex set $C$ such that $C\subseteq\ran f\subseteq\overline{C}.$ Then $\overline{C}=\R^n.$ As $C$ is convex, by \cite[Theorem 6.3]{convanalrock} we have
the relative interior
$\ri\overline{C}=\ri C,$ so $\ri C=\R^n.$ Thus, $\R^n=\ri C\subseteq C,$
which gives us that $C=\R^n.$ Therefore, $\ran\partial f=\R^n.$ By Fact \ref{inverse},
$\ran\partial f\subseteq\dom(f^*).$ Hence, $\dom f^*=\R^n.$ By Fact \ref{domfcoercive}, we have that $\lim\limits_{\|x\|\rightarrow\infty}\frac{f(x)}{\|x\|}=\infty.$ Therefore, $f$ is coercive for all $f\in G.$ Since $G$ is generic in $X$ and $G\subseteq S,$ we conclude that $S$ is generic in $X.$
\end{proof}

\begin{thm}\label{t:fulldom}
 In $(\Gamma_{0}(\R^n), d)$, the set
$S:=\{f\in \Gamma_{0}(\R^n):\ \dom f=\R^n \}$
is generic.
\end{thm}
\begin{proof} Note that $(\Gamma_{0}(\R^n))^*=\Gamma_{0}(\R^n)$. In $((\Gamma_{0}(\R^n))^*, d)$,
by Theorem~\ref{t:fullrange}, the set
$$\{f^*\in (\Gamma_{0}(\R^n))^*:\ f^* \text{ is  coercive} \}$$
is generic. Since $f^*$ is coervcive if and only if $f$ has $\dom f=\R^n$ by Fact~\ref{domfcoercive},
the proof is done.
\end{proof}

Combining Theorems~\ref{generictheorem}, \ref{t:fullrange} and ~\ref{t:fulldom}, we obtain
\begin{cor} In $(\Gamma_{0}(\R^n), d)$, the set
$$S:=\{f\in \Gamma_{0}(\R^n):\ \dom f=\R^n, \dom f^*=\R^n, f \text{ has a strong minimizer}\}$$
is generic.
\end{cor}

%%%%%
\section{Conclusion}\label{sec:conc}
%%%%%

 Endowed with the Attouch-Wets metric, based on the Moreau envelope,
 the set of proper lower semicontinuous convex functions becomes a complete metric space. In this complete metric space,
 the topology  is epi-convergence topology.
 We have proved several Baire category results. In particular, we have shown that in $(\Gamma_0(R^n),d)$ the set of strongly convex functions is category one, the set of functions that attain a strong minimum is category two, and the set of coercive functions is category two. Several other results about strongly convex functions and functions with strong minima are included. In future work that has already commenced, we will continue to develop the theory of Moreau envelopes,
 providing characterizations and illustrative examples of how to calculate them, and extend results in
this paper to convex functions defined on Hilbert spaces or to prox-bounded functions on $\R^n$.

\bibliographystyle{plain}
\bibliography{Bibliography}{}
\end{document}